\documentclass[reqno,11pt,centertags,draft]{amsart}
\usepackage[letterpaper,margin=1.5in,bottom=1.3in]{geometry}
\usepackage{amsmath,amsthm,amscd,amssymb,latexsym,enumerate}
\usepackage{overpic}
\usepackage{booktabs}
\usepackage{tikz}
\usetikzlibrary{arrows}
\usetikzlibrary{shapes.multipart}
\usepackage{caption}
\usepackage{subcaption}

\usepackage[final]{hyperref}

\allowdisplaybreaks

\newcommand*{\mailto}[1]{\href{mailto:#1}{\nolinkurl{#1}}}

\makeatletter
\def\theequation{\@arabic\c@equation}

\newcommand{\bbN}{{\mathbb{N}}}
\newcommand{\bbR}{{\mathbb{R}}}

\newcommand{\bbZ}{{\mathbb{Z}}}
\newcommand{\bbC}{{\mathbb{C}}}

\newcommand{\cC}{{\mathcal C}}

\newcommand{\no}{\nonumber}
\newcommand{\lb}{\label}
\newcommand{\bi}{\bibitem}
\newcommand{\f}{\frac}

\newcommand{\dott}{\,\cdot\,}

\newcommand{\Oh}{O}

\renewcommand{\Re}{\operatorname{Re}}
\renewcommand{\Im}{\operatorname{Im}}
\renewcommand{\ln}{\operatorname{ln}}

\numberwithin{equation}{section}

\newtheorem{theorem}{Theorem}[section]

\newtheorem{proposition}[theorem]{Proposition}

\newtheorem{example}[theorem]{Example}

\theoremstyle{definition}

\newtheorem{remark}[theorem]{Remark}

\usepackage{graphicx}
\DeclareGraphicsExtensions{.pdf}

\begin{document}
\title[Meijer's $G$-Function and Euler's Differential Equation]{Meijer's $G$-Function and Euler's Differential Equation Revisited}

\author{Fritz Gesztesy}
\address{Department of Mathematics,
Baylor University, Sid Richardson Bldg., 1410 S.\,4th Street,
Waco, TX 76706, USA}
\email{\mailto{Fritz\_Gesztesy@baylor.edu}}
\urladdr{\url{https://math.artsandsciences.baylor.edu/person/fritz-gesztesy-phd}}

\author{Markus Hunziker}
\address{Department of Mathematics,
Baylor University, Sid Richardson Bldg., 1410 S.\,4th Street,
Waco, TX 76706, USA}
\email{\mailto{Markus\_Hunziker@baylor.edu}}
\urladdr{\url{https://math.artsandsciences.baylor.edu/person/markus-hunziker-phd}}

\date{\today}

\begin{abstract}
We consider the generalized eigenvalue problem for the classical Euler differential equation and demonstrate its intimate connection with Meijer's $G$-functions. 

In the course of deriving the solution of the generalized Euler eigenvalue equation we review some of the basics of generalized hypergeometric functions and Meijer's $G$-functions and some of its special cases where the underlying Mellin-type integrand exhibits higher-order poles. 
\end{abstract}

\maketitle

{\scriptsize{\tableofcontents}}
\normalsize

\section{Introduction} \lb{s1}
 
While Euler's $N$th-order homogeneous differential equation, 
\begin{equation}
\Bigg(z^{-N} \prod_{j=1}^N \bigg[z\f{d}{dz} - \lambda_j\bigg]\Bigg) y(z) =0, \quad 
\lambda_j \in \bbC, \, 1 \leq j \leq N, \; z \in \bbC\backslash\{0\},     \lb{1.1} 
\end{equation}
and its nonhomogeneous version,
\begin{equation}
\Bigg(z^{-N} \prod_{j=1}^N \bigg[z\f{d}{dz} - \lambda_j\bigg]\Bigg) y(z) = f(z), \quad 
\lambda_j \in \bbC, \, 1 \leq j \leq N, \; z \in \bbC\backslash\{0\},      \lb{1.2} 
\end{equation}
has received enormous attention in the ODE literature, see, for instance, \cite[p.~122--123, 131, ]{CL85}, \cite[p.~85]{Ha82}, \cite[p.343--344]{Hi97}, \cite[p.~141--143, 202]{In56}, \cite[p.~26--27, 53--54]{Po36}, \cite[p.~116--117]{Te12}, \cite[p.~208--209]{Wa98}, and our observations in Section \ref{s5}, the associated generalized eigenvalue differential equation, 
\begin{equation}
\Bigg(z^{-N} \prod_{j=1}^N \bigg[z\f{d}{dz} - \lambda_j\bigg]\Bigg) y(\mu,z)= \mu y(\mu,z), \quad \lambda_j \in \bbC, \, 1 \leq j \leq N, \; z \in \bbC\backslash\{0\},   \lb{1.3} 
\end{equation}
with $\mu \in \bbC$ the spectral parameter, seems to have escaped scrutiny, at least, we were not able to find pertinent references. 

This paper aims at filling this gap and demonstrate that the solutions of \eqref{1.3} can be expressed in terms of Meijer's $G$-function, and, generically in terms of generalized hypergeometric functions (see Theorems \ref{t4.1} and \ref{t4.3}).  

In Section \ref{s2} we recall some of the basic facts on generalized hypergeometric functions and Meijer's $G$-functions, relevant in our context and, in the generic case of the parameters involved (see \eqref{2.5} and \eqref{2.15},  describe a fundamental system of solutions for the generalized hypergeometric differential equation \eqref{2.4} and a fundamental system of the Meijer $G$-function differential equation \eqref{2.10}.

The case of fundamental systems in the general (non-generic) case for Meijer's $G$-function $G^{m,n}_{p,q}$ is incredibly complex and beyond the scope of this note. Instead, in Section \ref{s3}, we only discuss the special situation $G^{m,0}_{0,q}$, a case that is tailored precisely to the generalized Euler eigenvalue problem at hand. In particular, we permit higher order poles in the integrand in the Mellin-type integral for $G^{m,0}_{0,q}$. We have not been able to locate this result, even in the present special case considered, in the literature. 

The results of Section \ref{s3} are then applied in our principal Section \ref{s4} to derive the solutions of \eqref{1.3}. 

Finally, in Section \ref{s5} we add some remarks on the nonhomogeneous $N$th-order Euler differential equation \eqref{1.2}.

\section{Basics on Generalized Hypergeometric Functions and Meijer's $G$-Functions} \lb{s2}

In this preparatory section we recall some of the basic properties of generalized hypergeometric functions and Meijer's $G$-functions. For details on generalized hypergeometric functions and Meijer's $G$-functions we refer, for instance, to \cite{BS13}, \cite[Ch.~IV, Sects.~5.3--5.6]{EMOT53}, \cite[Ch.~V]{Lu69}, \cite[Ch.~V]{Lu75}, \cite[Chs.~I, II, IV]{MS73}, \cite[Ch.~16]{OLBC10}, and \cite[Sect.~8.2]{PBM90}.

We start with generalized hypergeometric functions. Introducing\footnote{We employ the convenient abbreviation $\bbN_0= \bbN \cup \{0\}$.}
\begin{align}
\begin{split} 
& {}_r F_s\left(\!\begin{array}{c} {\displaystyle \alpha_1,\dots,\alpha_r} \\
{\displaystyle \beta_1,\dots,\beta_s}\end{array} \bigg\vert\, \zeta \right) = \sum_{k \in \bbN_0} 
\f{\prod_{\ell=1}^r (\alpha_{\ell})_k}{\prod_{m=1}^s (\beta_{m})_k} \, \f{\zeta^k}{k!}, \quad \zeta \in \bbC, \\
& r, s \in \bbN_0, \; r \leq s, \; \alpha \in \bbC, \, 1 \leq \ell \leq r, \; \beta_m \in \bbC \backslash \{- \bbN_0\}, \, 1 \leq m \leq s,   \lb{2.1} 
\end{split} 
\end{align}
with
\begin{equation}
(\gamma)_0 = 1, \; (\gamma)_k = \gamma (\gamma + 1) \cdots (\gamma + k - 1) 
= \Gamma(\gamma + k)/\Gamma(\gamma), \; k \in \bbN, \; \gamma \in \bbC,      \lb{2.2} 
\end{equation}
denoting Pochhammer's symbol. 

Then ${}_r F_s\Big(\!\begin{array}{c} {\scriptstyle \alpha_1,\dots,\alpha_r} \\
{\scriptstyle \beta_1,\dots,\beta_s}\end{array} \Big\vert\, \cdot \Big)$, $r \leq s$, is entire, 
\begin{equation}
{}_r F_s\left(\!\begin{array}{c} {\displaystyle \alpha_1,\dots,\alpha_r} \\
{\displaystyle \beta_1,\dots,\beta_s}\end{array} \bigg\vert\, \zeta \right) \underset{\zeta \to 0}{=} 
1 +\Oh(\zeta),      \lb{2.3}
\end{equation}
and it satisfies the differential equation of order $s+1$ (see  \cite[p.~184]{EMOT53}, \cite[p.~136--137]{Lu69}, \cite[p.~190]{Lu75}), 
\begin{equation}
\Bigg[\zeta \f{d}{d\zeta} \prod_{m=1}^s \bigg(\zeta \f{d}{d\zeta} + \beta_m - 1\bigg)  
- \zeta \prod_{\ell=1}^r \bigg(\zeta \f{d}{d\zeta} + \alpha_{\ell}\bigg)\Bigg] u(\zeta) = 0, \quad \zeta \in \bbC\backslash\{0\}. 
\lb{2.4} 
\end{equation}
Since $r \leq s$, $\zeta = 0$ is a regular singularity and $\zeta = \infty$ is an irregular (i.e., essential) singularity  of \eqref{2.4}; there are no other singularities of \eqref{2.4}. 

If one strengthens the conditions on $\alpha_{\ell}$, $1 \leq \ell \leq r$, $\beta_m$, $1 \leq m \leq s$, in \eqref{2.1} to
\begin{align} 
\begin{split} 
& r, s \in \bbN_0, \; r \leq s, \; \alpha_{\ell} \in \bbC, \, 1 \leq \ell \leq r, \; \beta_m \in \bbC \backslash \{- \bbN_0\}, \, 1 \leq m \leq s,    \\
& \beta_m - \beta_{m'} \in \bbC \backslash \bbZ, \quad 1 \leq m, m' \leq s, \; m \neq m',     \lb{2.5} 
\end{split} 
\end{align}
the following $s+1$ functions 
\begin{align}
& u_0(\zeta) = {}_r F_s\left(\!\begin{array}{c} {\displaystyle \alpha_1,\dots,\alpha_r} \\
{\displaystyle \beta_1,\dots,\beta_s}\end{array} \bigg\vert\, \zeta \right),   \no \\
& u_m(\zeta) = \zeta^{1-\beta_m}      \lb{2.6} \\
& \quad \times {}_r F_s\left(\!\begin{array}{c} {\displaystyle 1+\alpha_1-\beta_m,\dots\dots\dots\dots\dots
\dots. \, .\dots\dots\dots\dots\dots\dots,1+\alpha_r-\beta_m} \\
{\displaystyle 1+\beta_1-\beta_m,\dots,1+\beta_{m-1}-\beta_m,1+\beta_{m+1}-\beta_m,\dots,1+\beta_s-\beta_m}\end{array} \bigg\vert\, \zeta \right),    \no \\
& \hspace*{10cm} 1 \leq m \leq s; \; \zeta \in \bbC\backslash\{0\},    \no
\end{align}
constitute a fundamental system of solutions of \eqref{2.4} (see, \cite[p.~137]{Lu69}, \cite[p.~190]{Lu75}, and \cite[p.~409]{OLBC10}). In general, the solutions $u_m$, $ 1 \leq m \leq s$, in \eqref{2.5} acquire the branch cut inherited from $\zeta \mapsto \zeta^{1-\beta_m}$.

The setup in \eqref{2.4}--\eqref{2.6} formally corresponds to introducing $\beta_0 = 1$ in the first term of \eqref{2.4}. 

\begin{remark} \lb{r2.1}
$(i)$ Our notation in \eqref{2.6} (and later on) is somewhat defective: For instance, the case $m=1$ in \eqref{2.6} leads to nonexistent terms as the $\beta$-row would necessarily start with $1+\beta_2 - \beta_1$. But we prefer this ambiguity over introducing various case distinctions, unnecessarily complicating the presentation. \\[1mm] 
$(ii)$ The restriction $\beta_m \in \bbC \backslash \{0\}$, $1 \leq m \leq s$, in \eqref{2.1} is necessitated by the fact that ${}_r F_s\Big(\!\begin{array}{c} {\scriptstyle \alpha_1,\dots,\alpha_r} \\
{\scriptstyle \beta_1,\dots,\beta_s}\end{array} \Big\vert\, \cdot \Big)$ exhibits a first order pole as $\beta_m \to -n$, $n \in \bbN_0$. This yields a simple recipe to define a closely related function, ${}_r {\bf F}_s\Big(\!\begin{array}{c} {\scriptstyle \alpha_1,\dots,\alpha_r} \\ {\scriptstyle \beta_1,\dots,\beta_s}\end{array} \Big\vert\, \cdot \Big)$, which is well-defined for all 
$\beta_m \in \bbC$, $1 \leq m \leq s$, viz.,
\begin{align}
{}_r {\bf F}_s\Big(\!\begin{array}{c} {\scriptstyle \alpha_1,\dots,\alpha_r} \\ {\scriptstyle \beta_1,\dots,\beta_s}\end{array} \Big\vert\, \cdot \Big) &= {}_r F_s\Big(\!\begin{array}{c} {\scriptstyle \alpha_1,\dots,\alpha_r} \\ {\scriptstyle \beta_1,\dots,\beta_s}\end{array} \Big\vert\, \cdot \Big) \bigg/ \prod_{m=1}^s \Gamma(\beta_m)      \no \\
&= \sum_{k \in \bbN_0} 
\f{\prod_{\ell=1}^r (\alpha_{\ell})_k}{\prod_{m=1}^s \Gamma(\beta_{m} + k)} \, \f{\zeta^k}{k!}, \quad \zeta \in \bbC, \\
& \hspace*{-2.15cm}
r, s \in \bbN_0, \; r \leq s, \; \alpha, \, 1 \leq \ell \leq r, \; \beta_m \in \bbC, \, 1 \leq m \leq s,   \no
\end{align}
a device used, for instance, in \cite[Sect.~5.9]{Ol97} and \cite[Nos.~ 15.2.2, 16.2.5]{OLBC10}. 
\hfill $\diamond$
\end{remark}

For additional background information in connection with generalized hypergeometric functions  we refer, for instance, to the monographs, \cite[Ch.~11]{An98}, \cite[Ch.~II]{Ba64}, \cite[Ch.~IV]{EMOT53}, \cite[Chs.~V--VII]{Lu69}, \cite[Ch.~IX]{Lu69a}, \cite[Ch.~V, VIII]{Lu75}, \cite[Sect.~2.3]{MS73}, \cite[Ch.~16]{OLBC10}, \cite[Sect.~7.2]{PBM90}, \cite[Ch.~2]{Sl66}.  

Next, we turn to some of the basics of Meijer's $G$-functions and hence introduce
\begin{align}
& G_{p,q}^{m,n}\left(\!\begin{array}{c} {\displaystyle a_1,\dots,a_n; a_{n+1},\dots,a_p} \\
{\displaystyle b_1,\dots,b_m; b_{m+1},\dots,b_q}\end{array} \bigg\vert\, \zeta \right)     \no \\
& \quad = \f{1}{2\pi i} 
\int_{\cC_{b_1,\dots,b_m}} d \omega \, \f{\prod_{j=1}^m \Gamma(b_j-\omega) \prod_{k=1}^n \Gamma(1-a_k+\omega)}
{\prod_{j=m+1}^q \Gamma(1-b_j+\omega) \prod_{k=n+1}^p \Gamma(a_k-\omega)} \, \zeta^{\omega},   
\lb{2.7} \\
& \hspace*{3cm} 0 \leq n \leq p, \; 0 \leq m \leq q, \; q \geq 1, \; p < q, \; \zeta \in \bbC\backslash\{0\},    \no
\end{align} 
assuming that no pole of $\Gamma(b_j - \cdot)$, $1 \leq j \leq m$, coincides with a pole of 
$\Gamma(1-a_k+\cdot)$, $1 \leq k \leq n$, thus,
\begin{equation}
a_k - b_j \in \bbC \backslash \bbN, \quad 1 \leq j \leq m, \; 1 \leq k \leq n.      \lb{2.8} 
\end{equation}
Here 
\begin{align}
& \text{$\cC_{b_1,\dots,b_m}$ is a contour beginning and ending at $+ \infty$,} \no \\
& \quad \text{encircling all poles of 
$\Gamma(b_j-\cdot)$, $1 \leq j \leq m$, once in negative orientation,}  \lb{2.8a} \\
& \quad \text{but none of the poles of 
$\Gamma(1-a_k+\cdot)$, $1 \leq k \leq n$,}   \no 
\end{align} 
and the left-hand side of \eqref{2.7} is defined as the (absolutely convergent) sum of residues of the right-hand side. 

As a result, \eqref {2.7} is trivial if $m=0$,
\begin{equation}
G_{p,q}^{0,n}\left(\!\begin{array}{c} {\displaystyle a_1, \dots,a_n;a_{n+1},\dots ,a_p} \\
{\displaystyle ; b_1, . \dots \dots \dots \dots \dots ,b_q}\end{array} \bigg\vert\, \dott \right) = 0.      \lb{2.9}
\end{equation}

One verifies that $G_{p,q}^{m,n}\Big(\!\begin{array}{c} {\scriptstyle a_1,\dots,a_n; a_{n+1},\dots,a_p} \\
{\scriptstyle b_1,\dots,b_m; b_{m+1}, \dots, b_q}\end{array} \Big\vert\, \cdot \Big)$, $p < q$, satisfies the differential equation of order $q$ (see  \cite[p.~210]{EMOT53}, \cite[p.~181]{Lu69}, \cite[p.~192]{Lu75}, \cite[p.~13, 16]{MS73}),
\begin{equation}
\Bigg[\prod_{j=1}^q \bigg(\zeta \f{d}{d\zeta} - b_j\bigg) 
- (-1)^{m+n-p}\zeta \prod_{k=1}^p \bigg(\zeta \f{d}{d\zeta} - a_k + 1\bigg)\Bigg] u(\zeta) = 0, 
\quad \zeta \in \bbC\backslash\{0\}.       \lb{2.10} 
\end{equation}
Again, since $p < q$, $\zeta = 0$ is a regular singularity and $\zeta = \infty$ is an irregular (i.e., essential) singularity  of \eqref{2.10}; there are no other singularities of \eqref{2.10}. 

Next, if 
\begin{equation}
b_j - b_{j'} \in \bbC \backslash \bbZ, \quad 1 \leq j, j' \leq m, \; j \neq j',    \lb{2.11} 
\end{equation}
then (cf.\ \cite[p.~208]{EMOT53}, \cite[p.~145]{Lu69}, \cite[p.~171]{Lu75}, \cite[p.~3--4]{MS73})
\begin{align}
& G_{p,q}^{m,n}\left(\!\begin{array}{c} {\displaystyle a_1,\dots,a_n; a_{n+1},\dots,a_p} \\
{\displaystyle b_1,\dots,b_m; b_{m+1},\dots,b_q}\end{array} \bigg\vert\, \zeta \right)    \no \\
& \quad = \sum_{j=1}^m \f{\prod_{j'=1 \, j' \neq j}^m \Gamma(b_{j'} -b_j) \prod_{k=1}^n \Gamma(1+b_j-a_k)}{\prod_{j' = m+1}^q\Gamma(1+b_j-b_{j'}) \prod_{k=n+1}^p \Gamma(a_k-b_j)} \, \zeta^{b_j}    \lb{2.12} \\
& \qquad \times {}_p F_{q-1}\bigg(\!\begin{array}{c} {\displaystyle 1+b_j-a_1,\dots\dots\dots\dots\dots\,.\,
\,.\,\,.\,\dots\dots\dots\dots\dots,1+b_j-a_p} \\
{\displaystyle 1+b_j-b_1,\dots,1+b_{j}-b_{j-1},1+b_j-b_{j+1},\dots,1+b_j-b_q}\end{array} \bigg\vert   \no \\
& \hspace*{10.7cm} \bigg\vert\, (-1)^{p-m-n} \zeta \bigg),    \no \\ 
& \hspace*{5.25cm} 0 \leq n \leq p, \; 0 \leq m \leq q, \; q \geq 1, \; p < q, \; \zeta \in \bbC\backslash\{0\}.    \no 
\end{align}
In particular, if $m=1$ (and hence \eqref{2.11} becomes irrelevant), then (cf.\ \cite[p.~174]{Lu75}, \cite[p.~4]{MS73})
\begin{align}
& G_{p,q}^{1,n}\left(\!\begin{array}{c} {\displaystyle a_1, . . . \dots ,a_n; a_{n+1}, \dots . . . ,a_p} \\
{\displaystyle b_j;b_1,\dots,b_{j-1}, b_{j+1},\dots,b_q}\end{array} \bigg\vert\, \zeta \right)    \no \\
& \quad = \f{\prod_{k=1}^n \Gamma(1+b_j-a_k)}{\prod_{j' = 1}^q\Gamma(1+b_j-b_{j'}) \prod_{k=n+1}^p \Gamma(a_k-b_j)} \, \zeta^{b_j}     \lb{2.13} \\
& \qquad \times {}_p F_{q-1}\bigg(\!\begin{array}{c} {\displaystyle 1+b_j-a_1,\dots\dots\dots\dots\dots\,.\,
\,.\,\,.\,\dots\dots\dots\dots\dots,1+b_j-a_p} \\
{\displaystyle 1+b_j-b_1,\dots,1+b_{j}-b_{j-1},1+b_j-b_{j+1},\dots,1+b_j-b_q}\end{array} \bigg\vert   \no \\
& \hspace*{10.8cm} \bigg\vert\, (-1)^{p-1-n} \zeta \bigg),    \no \\ 
& \hspace*{6.55cm} 0 \leq n \leq p, \; 1 \leq j \leq q, \; p < q, \; \zeta \in \bbC\backslash\{0\},     \no 
\end{align}
and (still assuming \eqref{2.11}),
\begin{align}
& G_{p,q}^{m,n}\left(\!\begin{array}{c} {\displaystyle a_1,\dots,a_n; a_{n+1},\dots,a_p} \\
{\displaystyle b_1,\dots,b_m; b_{m+1},\dots,b_q}\end{array} \bigg\vert\, \zeta \right)    \no \\
& \quad = \sum_{j=1}^m \f{\prod_{j'=1 \, j' \neq j}^m \Gamma(b_{j'} -b_j) \Gamma(1+b_j-b_{j'})}{\prod_{k = n+1}^p \Gamma(1+b_j-a_{j'}) \Gamma(a_k-b_j)} \, e^{-i \pi b_j(p+1-m-n)}    \lb{2.14} \\
& \qquad \times G_{p,q}^{1,p}\left(\!\begin{array}{c} {\displaystyle a_1, . . \dots \dots \dots \dots \dots \dots . ,a_p;} \\
{\displaystyle b_j;b_1,\dots,b_{j-1}, b_{j+1},\dots,b_q}\end{array} \bigg\vert\, (-1)^{p+1-m-n} \zeta \right) ,    \no \\ 
& \hspace*{2.2cm} 0 \leq n \leq p, \; 0 \leq m \leq q, \; q \geq 1, \; p < q, \; \zeta \in \bbC\backslash\{0\}    \no 
\end{align}
(see, \cite[p.~175]{Lu75}, \cite[p.~14]{MS73}).

\begin{remark} 
Given Remark \ref{r2.1}\,$(ii)$, the presence of the factors $\prod_{j'=m+1}^q \Gamma(1+b_j-b_{j'})$, respectively, 
$\prod_{j'=1}^q \Gamma(1+b_j-b_{j'})$, in the denominators of the right-hand sides of \eqref{2.12}, respectively, \eqref{2.13}, renders these right-hand sides well-defined without imposing the additional condition 
$b_j - b_{j'} \in \bbC \backslash \bbZ$, $1\leq j \leq m$, $m+1 \leq j' \leq q$. 
\hfill $\diamond$
\end{remark} 

We also recall that if 
\begin{equation}
b_j - b_{j'} \in \bbC \backslash \bbZ, \quad 1 \leq j, j' \leq q, \; j \neq j',    \lb{2.15} 
\end{equation}
then (cf.\ \cite[p.~181]{Lu69}, \cite[p.~190]{Lu75}, \cite[p.~621--622]{PBM90})
\begin{align}
& G_{p,q}^{1,p}\left(\!\begin{array}{c} {\displaystyle a_1, . . \dots \dots \dots \dots \dots \dots . ,a_p;} \\
{\displaystyle b_j;b_1,\dots,b_{j-1}, b_{j+1},\dots,b_q}\end{array} \bigg\vert\, \zeta \right)    \no \\
& \quad = \f{\prod_{k=1}^p \Gamma(1+b_j-a_k)}{\prod_{j' = 1}^q\Gamma(1+b_j-b_{j'})} \, \zeta^{b_j}     \lb{2.16} \\
& \qquad \times {}_p F_{q-1}\bigg(\!\begin{array}{c} {\displaystyle 1+b_j-a_1,\dots\dots\dots\dots\dots\,.\,
\,.\,\,.\,\dots\dots\dots\dots\dots,1+b_j-a_p} \\
{\displaystyle 1+b_j-b_1,\dots,1+b_{j}-b_{j-1},1+b_j-b_{j+1},\dots,1+b_j-b_q}\end{array} \bigg\vert   \no \\
& \hspace*{12cm} \bigg\vert - \zeta \bigg),    \no \\ 
& \hspace*{8.4cm} 1 \leq j \leq q, \; p < q, \; \zeta \in \bbC\backslash\{0\},     \no 
\end{align}
constitutes a fundamental system of solutions of \eqref{2.10}, subject to \eqref{2.15}. Equivalently, taking into account the asymptotics \eqref{2.3}, 
\begin{align}
& \zeta^{b_j} {}_p F_{q-1}\bigg(\!\begin{array}{c} {\displaystyle 1+b_j-a_1,\dots\dots\dots\dots\dots\,.\,
\,.\,\,.\,\dots\dots\dots\dots\dots,1+b_j-a_p} \\
{\displaystyle 1+b_j-b_1,\dots,1+b_{j}-b_{j-1},1+b_j-b_{j+1},\dots,1+b_j-b_q}\end{array} \bigg\vert - \zeta \bigg),  \no \\
& \hspace*{7.2cm} 1 \leq j \leq q, \; p < q, \; \zeta \in \bbC\backslash\{0\},     \lb{2.17} 
\end{align}
represents a fundamental system of solutions of \eqref{2.10}, subject to \eqref{2.15} (see, \cite[p.~181]{Lu69}).

The case when condition \eqref{2.15} is violated is much more involved and we will study a special case, particularly suited for our application to Euler's differential equation, in the next section.

For additional results on Meijer's $G$-function see, for instance, \cite[Sect.~11.3]{An98}, 
\cite{BS13}, \cite{DKK17}, \cite[Ch.~V]{EMOT53}, \cite{Fi72}, \cite[Chs.~V, VI]{Lu69}, \cite[Ch.~IX]{Lu69a}, \cite[Ch.~V, VIII]{Lu75}, \cite{MS73}, \cite{Me36}--\cite{Me52}, \cite[Ch.~16]{OLBC10}, \cite[Sect.~8.2]{PBM90}, \cite{Ro82/83}, \cite[Ch.~2]{Sl66}, \cite{Wi22}, \cite{Wo01}.

\section{Meijer's $G$-Functions in the Presence of Higher-Order Poles} \lb{s3}

In this section we take a close look at a fundamental system of solutions of then differential equation \eqref{2.10} in the absence of condition \eqref{2.11} (surprisingly, we were not able to locate the details in the literature). In this case the integrand on the right-hand side of \eqref{2.7} exhibits higher-order poles and finding a fundamental set of solutions of \eqref{2.10} becomes a rather complex task.  

To analyze the right-hand side of \eqref{2.7} in the case the integrand has higher-order poles we start with some preparations: \\[1mm] 
${\bf (I)}$ Let $c \in \bbC$. Then $\Gamma(c - \dott)$ has the partial fraction expansion (see, e.g., \cite[p.~162]{Ti85})
\begin{equation}
\Gamma(c - \omega) = \underbrace{\int_1^{\infty} dx \, e^{-x} x^{c-1- \omega}}_{\text{entire w.r.t. $\omega$}}
+ \sum_{n \in \bbN_0} \f{(-1)^{n+1}}{n! (\omega - c -n)}, 
\quad \omega \in \bbC \backslash \{c + \bbN_0\}.      \lb{3.1}
\end{equation}
${\bf (II)}$ Let $b \in \bbC$. Then $\Gamma(b - \omega)$ has first order poles precisely at the points $\omega \in b + \bbN_0$ and is analytic otherwise on $\bbC$. The residues of $\Gamma(b - \omega)$ at $\omega \in b+\bbN_0$, are given by 
$(-1)^{n+1}/\Gamma(n+1)$, that is,
\begin{equation}
\lim_{\omega \to b+n} (\omega - b -n) \Gamma(b - \omega) = (-1)^{n+1}/(n!), \quad n \in \bbN_0.     \lb{3.2}
\end{equation}
(We use the usual convention $0! = 1$.) \\[1mm] 
${\bf (III)}$ Let $b \in \bbC$. Then $1/\Gamma(1 - b + \omega)$ is entire with respect to $\omega$ with simple zeros at $\omega \in b - \bbN$; there are no other zeros in $\bbC$. 

Consequently, the poles of $\Gamma(b_j - \dott)$ and zeros of $1/\Gamma(1 - b_{j'} + \dott)$, $1 \leq j, j' \leq q$, cannot interfere (and possibly cancel each other). \\[1mm] 
${\bf (IV)}$ The map $\bbC \ni \zeta \mapsto \zeta^{\omega} = e^{\omega \ln(\zeta)}$ has a cut in $\zeta$ (connecting $\zeta=0$ with infinity), but no cut in $\omega$. \\[1mm] 
${\bf (V)}$ Let $\Omega \subseteq \bbC$ be open and connected, $\omega_0 \in \Omega$, and $f : \Omega \backslash \{\omega_0\} \to \bbC$ analytic, with $\omega_0$ a pole of $f$ of order $k \in \bbN$. Then the residue of $f$ at $\omega_0$ is given by
\begin{equation}
{\rm Res\,}_{\omega_0} (f(\dott)) = \lim_{\omega \to \omega_0} \f{1}{(k-1)!} \f{d^{k-1}}{d\omega^{k-1}} \big[(\omega - \omega_0)^k f(\omega)\big].      \lb{3.3}
\end{equation}
${\bf (VI)}$ For subsequent use we also note the elementary fact,
\begin{equation}
\f{d^{k-1}}{d \omega^{k-1}} \zeta^{\omega} = \zeta^{\omega} [\ln(\zeta)]^{k-1}, \quad \zeta \in \bbC \backslash \{0\}, \; k \in \bbN.     \lb{3.3a}
\end{equation}

Next, going beyond condition \eqref{2.15} we decompose the multiset $\{b_1,\dots,b_q\}$ into congruence classes ${\rm mod}\,\bbZ$ as follows: 
\begin{align}
\begin{split} 
& \{b_1,\dots,b_q\} = \{\underbrace{b_1,\dots,b_{n_1}}_{\text{$n_1$ terms}}\} \cup 
\{\underbrace{b_{n_1+1},\dots,b_{n_1+n_2}}_{\text{$n_2$ terms}}\} \cup \cdots      \lb{3.4} \\
& \hspace*{2.5cm} \cdots \cup \{\underbrace{b_{n_1+\cdots + n_{r-1}+1},\dots,b_{n_1+\cdots +n_{r-1} + n_r} \equiv b_q}_{\text{$n_r$ terms}}\},  
\end{split} \\
& \sum_{s=1}^r n_s = q,    \lb{3.5} \\
\begin{split}
& \Re(b_1) \geq \cdots \geq \Re(b_{n_1}), \\
& \Re(b_{n_1+1}) \geq \cdots \geq \Re(b_{n_1+n_2}), \\
& \quad \vdots      \lb{3.6} \\
&  \Re(b_{n_1+\cdots n_{r-1}+1}) \geq \cdots \geq \Re(b_{n_1+\cdots n_{r-1}+n_r}), 
\end{split} \\
\begin{split} 
& b_j - b_{j'} \in \bbZ, \; 1 \leq j, j' \leq n_1, \\
& b_{n_1+j} - b_{n_1+j'} \in \bbZ, \; 1 \leq j, j' \leq n_2,  \\
& \quad \vdots    \lb{3.7} \\
& b_{n_1+\cdots+n_{r-1}+j} - b_{n_1+\cdot+n_{r-1}+j'} \in \bbZ, \; 1 \leq j, j' \leq n_r.
\end{split}
\end{align}

One notes that 
\begin{equation}
G_{0,q}^{m,0}\left(\!\begin{array}{c} {\displaystyle } \\
{\displaystyle b_1,\dots,b_m; b_{m+1},\dots,b_q}\end{array} \bigg\vert\, (-1)^m \zeta \right), \quad 1 \leq m \leq q, \lb{3.8}
\end{equation}
satisfies the ordinary differential equation of order $q$, 
\begin{equation}
\Bigg[\prod_{j=1}^q \bigg(\zeta \f{d}{d\zeta} - b_j\bigg) - \zeta \Bigg] u(\zeta) = 0, 
\quad \zeta \in \bbC\backslash\{0\}.       \lb{3.9} 
\end{equation}

\noindent 
{\bf Case (1):} In the generic case, where all congruence classes consist of precisely one element, equivalently, when $r=q$, then linearly independent solutions, in fact, a fundamental system of solutions of \eqref{3.9}, is given by
\begin{align}
& G_{0,q}^{1,0}\left(\!\begin{array}{c} {\displaystyle } \\
{\displaystyle b_j; b_1,\dots, b_{j-1},b_{j+1},\dots,b_q}\end{array} \bigg\vert - \zeta \right) 
= \f{\zeta^{b_j}}{\prod_{j'=1}^q \Gamma(1+b_j-b_{j'})}     \no \\
& \qquad \times {}_0 F_{q-1}\bigg(\!\begin{array}{c} {\displaystyle } \\
{\displaystyle 1+b_j-b_1,\dots,1+b_{j}-b_{j-1},1+b_j-b_{j+1},\dots,1+b_j-b_q}\end{array} \bigg\vert\, \zeta \bigg),  \no \\
& \quad \underset{\zeta \to 0}{=} C_j \zeta^{b_j}[1 + \Oh(\zeta)], \quad C_j \in \bbC \backslash \{0\}, \; 1 \leq j \leq q, \; \zeta \in \bbC\backslash\{0\},   \lb{3.10}
\end{align}
a special case of the situation mentioned in Section \ref{s3}. \\[2mm] 
\noindent 
{\bf Case (2):} It suffices to consider the congruence class $\{b_1,\dots,b_{n_1}\}$ in some detail as the remaining ones are treated in exactly the same manner. The existence of a system of linearly independent solutions of \eqref{3.9} will be demonstrated by the different asymptotic behavior of solutions as $\zeta \to 0$. The treatment of all congruence classes in \eqref{3.4} then yields a fundamental system of solutions of \eqref{3.9}. 

Referring to \eqref{3.8}, one obtains the system of solutions, 
\begin{align}
y_1(\zeta) &= G_{0,q}^{1,0}\left(\!\begin{array}{c} {\displaystyle } \\
{\displaystyle b_1; b_2,\dots,b_q}\end{array} \bigg\vert - \zeta \right)    \no \\
&= \f{\zeta^{b_1}}{\prod_{j'=1}^q \Gamma(1-b_1-b_{j'})} \, {}_0 F_{q-1}\bigg(\!\begin{array}{c} {\displaystyle } \\
{\displaystyle 1+b_1-b_2,\dots,1+b_1-b_q}\end{array} \bigg\vert\, \zeta \bigg),  \no \\
& \hspace*{-1.5mm} \underset{\zeta \to 0}{=} C_{1,1} \zeta^{b_1} [1 + \Oh(\zeta)],    \no \\ 
y_2(\zeta) &= G_{0,q}^{2,0}\left(\!\begin{array}{c} {\displaystyle } \\
{\displaystyle b_1, b_2; b_3\dots,b_q}\end{array} \bigg\vert\, \zeta \right)     \no \\
&  \hspace*{-1.5mm} \underset{\zeta \to 0}{=} C_{2,1} \zeta^{b_1} \ln(\zeta) [1 + \Oh(\zeta)] 
+ D_{2,1,b_1} \zeta^{b_1} [1 + \Oh(\zeta)] + C_{2,2} \zeta^{b_2} [1 + \Oh(\zeta)],     \no \\ 
&  \hspace*{-1.5mm} \underset{\zeta \to 0}{=} C_{2,1} \zeta^{b_1} \ln(\zeta) [1 + \Oh(\zeta)]  
+ C_{2,2} \zeta^{b_2} [1 + \Oh(\zeta)],     \no \\
& \hspace*{-6mm} \vdots    \no \\
y_{n_1}(\zeta) &= G_{0,q}^{n_1,0}\left(\!\begin{array}{c} {\displaystyle } \\
{\displaystyle b_1,\dots,b_{n_1}; b_{n_1+1},\dots,b_q}\end{array} \bigg\vert\, (-1)^{n_1} \zeta \right)     \no \\
& \hspace*{-1.5mm} \underset{\zeta \to 0}{=} C_{n_1,1} \zeta^{b_1} [\ln(\zeta)]^{n_1-1} [1 + \Oh(\zeta)]     
+ \zeta^{b_1} \sum_{k=2}^{n_1} D_{n_1,k,b_1} [\ln(\zeta)]^{n_1-k}[1 + \Oh(\zeta)]      \no \\ 
& \quad \; + C_{n_1,2} \zeta^{b_2} [\ln(\zeta)]^{n_1-2} [1 + \Oh(\zeta)] 
+ \zeta^{b_2} \sum_{k=3}^{n_1} D_{n_1,k,b_2} [\ln(\zeta)]^{n_1-k}[1 + \Oh(\zeta)]      \no \\ 
& \quad \; + \dots + C_{n_1,n_1} \zeta^{b_{n-1}} [1 + \Oh(\zeta)]      \no \\
& \hspace*{-1.5mm} \underset{\zeta \to 0}{=} C_{n_1,1} \zeta^{b_1} [\ln(\zeta)]^{n_1-1} [1 + \Oh(\zeta)]  
+ C_{n_1,2} \zeta^{b_2} [\ln(\zeta)]^{n_1-2} [1 + \Oh(\zeta)]     \no \\ 
& \quad \; + \dots + C_{n_1,n_1} \zeta^{b_{n_1}} [1 + \Oh(\zeta)]; \quad \zeta \in \bbC\backslash\{0\},   \lb{3.11}
\end{align}
where we used the ordering of the $b$'s in \eqref{3.6} to ignore all the $D$-terms corresponding to the coefficients $D_{2,1,b_1},\dots,D_{n_1,n_1,b_{n_1-1}}$. In addition, 
\begin{equation}
\text{\it all terms of the type $[1 + \Oh(\zeta)]$ are analytic at $\zeta = 0$,}      \lb{3.12}
\end{equation}
and 
\begin{align}
& C_{1,1} \in \bbC \backslash \{0\},   \no \\
& C_{2,2} \in \bbC \backslash \{0\} \, \text{ if $\Re(b_1) > \Re(b_2)$,} \quad C_{2,1} \in \bbC \backslash \{0\} \, 
\text{ if $\Re(b_1) = \Re(b_2)$ (i.e., $b_1=b_2$),}     \no \\
& C_{3,3} \in \bbC \backslash \{0\} \, \text{ if $\Re(b_2) > \Re(b_3)$,} \quad C_{3,2} \neq 0 \, 
\text{ if $\Re(b_1) > \Re(b_2) = \Re(b_3)$,}  \lb{3.13}    \no \\
& \quad C_{3,1} \in \bbC \backslash \{0\} \, \text{ if $\Re(b_1) = \Re(b_2) = \Re(b_3)$,}     \\
& \quad \vdots     \no \\
& C_{n_1,n_1} \in \bbC \backslash \{0\} \, \text{ if $\Re(b_{n_1-1}) > \Re(b_{n_1})$,} \dots     \no \\
& \quad \dots, C_{n_1,1} \in \bbC \backslash \{0\} \, \text{ if $\Re(b_1) = \Re(b_2) = \dots = \Re(b_{n_1})$.}   \no 
\end{align}
Clearly, the smallest $\Re(b_{\ell})$ on the right-hand sides of \eqref{3.11}, $1 \leq \ell \leq n_1$, together with the largest power of $\ln(\zeta)$ gives rise to the leading asymptotics as $\zeta \to 0$ in each of the solutions $y_j$, $1 \leq j \leq n_1$, in \eqref{3.11}. 

Taking \eqref{3.13} for granted, the asymptotic behavior as $\zeta \to 0$ in \eqref{3.11} shows that the latter represents a system of linearly independent solutions of \eqref{3.9}. 

To verify the statements in \eqref{3.12} and \eqref{3.13}, one relies on items ${\bf (I)}$--${\bf (VI)}$ as follows: Starting with the solution $y_1$, one analyzes 
\begin{equation}
\f{1}{2 \pi i} \int_{\cC_{b_1}} d \omega \, \f{\Gamma(b_1-\omega)}{\prod_{j=2}^q \Gamma(1-b_j+\omega)} (-\zeta)^{\omega},   
\end{equation} 
where the integrand has simple poles precisely at $\omega \in b_1+\bbN_0$ and is analytic in the region enclosed by $\cC_{b_1}$ and in a neighborhood of $\partial \cC_{b_1}$. Thus, one obtains from the residue theorem (taking into account that $\cC_{b_1}$ is negatively oriented)
\begin{align}
\begin{split} 
& \f{1}{2 \pi i} \int_{\cC_{b_1}} d \omega \, \f{\Gamma(b_1-\omega)}{\prod_{j=2}^q \Gamma(1-b_j+\omega)} (-\zeta)^{\omega}   \\
& \quad = - \sum_{n \in \bbN_0} {\rm Res\,}_{\omega= b_1+n} \bigg(\f{\Gamma(b_1-\omega)}{\prod_{j=2}^q \Gamma(1-b_j+\omega)} (-\zeta)^{\omega}\bigg)     \\
& \quad = - \sum_{n \in \bbN_0} \lim_{\omega \to b_1+n} \bigg[\f{(\omega - b_1-n) \Gamma(b_1-\omega)}{\prod_{j=2}^q \Gamma(1-b_j+\omega)} (-\zeta)^{\omega}\bigg]     \lb{3.15} \\
& \quad = - \sum_{n \in \bbN_0} \f{(-1)^{n+1}}{\Gamma(n+1)} \f{1}{\prod_{j=2}^q \Gamma(1-b_j+b_1+n)} (-\zeta)^{b_1+n}   \\
& \hspace*{2.5mm} \underset{\zeta \to 0}{=} \f{(-1)^{b_1}}{\prod_{j=2}^q \Gamma(1 + b_1 - b_j)} \zeta^{b_1} 
\underbrace{[1+\Oh(\zeta)]}_{\text{analytic at $\zeta=0$}} 
+ D \zeta^{b_1+1}\underbrace{[1+\Oh(\zeta)]}_{\text{analytic at $\zeta=0$}}      \\
& \hspace*{2.5mm} \underset{\zeta \to 0}{=} C_{1,1} \zeta^{b_1} [1+\Oh(\zeta)], 
\quad C_{1,1} \in \bbC \backslash \{0\}, \; D \in \bbC.   
\end{split}
\end{align} 
The fact that the $n=0$ term under the sum in \eqref{3.15} yields the leading asymptotic contribution as $\zeta \to 0$ is a general phenomenon that persists for the remaining solutions $y_2,\dots,y_{n_1}$, in fact, it applies to all solutions $y_1,\dots,y_q$ of \eqref{3.9}. 

Turning to the solution $y_2$, one now considers 
\begin{equation}
\f{1}{2 \pi i} \int_{\cC_{b_1,b_2}} d \omega \, \f{\Gamma(b_1-\omega) \Gamma(b_2-\omega)}{\prod_{j=2}^q \Gamma(1-b_j+\omega)} 
\zeta^{\omega},   \lb{3.16} 
\end{equation}
where the integrand now has poles at $\omega \in b_1+\bbN_0$ and $\omega \in b_2+\bbN_0$. More precisely, since $\Re(b_1) \geq \Re(b_2)$, one has $(b_1 - b_2) \in \bbN_0$, and thus $b_1$ and $b_2$ lie on a straight line parallel to the $\Im(\omega) = 0$-axis. In addition, the integrand in \eqref{3.16} has double (and simple) poles at $\omega \in b_1 + \bbN_0$, and simple poles at $\{b_2, b_2+1, \dots,b_2 +(b_1-b_2) - 1\}$, unless $\Re(b_1) = \Re(b_2)$ (and hence $b_1=b_2$) in which case the simple poles are absent.

For the simple poles (assuming temporarily that $\Re(b_1) > \Re(b_2)$) one closely follows the computation in \eqref{3.15}, and hence infers
\begin{align}
\begin{split} 
& \f{1}{2 \pi i} \int_{\cC_{b_1,b_2}} d \omega \, \f{\Gamma(b_1-\omega) \Gamma(b_2-\omega)}{\prod_{j=2}^q \Gamma(1-b_j+\omega)} 
\zeta^{\omega}   \\ 
& \quad = - \sum_{n=0}^{b_1-b_2-1} {\rm Res\,}_{\omega= b_2+n} \bigg(\f{\Gamma(b_1-\omega) \Gamma(b_2-\omega)}{\prod_{j=3}^q \Gamma(1-b_j+\omega)} \zeta^{\omega}\bigg)     \\ 
& \quad = - \sum_{n=0}^{b_1-b_2-1} \lim_{\omega \to b_2+n} 
\bigg[\f{[(\omega - b_2-n) \Gamma(b_2-\omega)] \Gamma(b_1-\omega)}{\prod_{j=3}^q \Gamma(1-b_j+\omega)} \zeta^{\omega}\bigg]    \lb{3.17} \\
& \hspace*{2.5mm} \underset{\zeta \to 0}{=} \f{\Gamma(b_1-b_2)}{\prod_{j=3}^q \Gamma(1 + b_1 - b_j)} \zeta^{b_2} 
\underbrace{[1+\Oh(\zeta)]}_{\text{analytic at $\zeta=0$}} 
+ D \zeta^{b_2+1}\underbrace{[1+\Oh(\zeta)]}_{\text{analytic at $\zeta=0$}}       \\
& \hspace*{2.5mm} \underset{\zeta \to 0}{=} C_{2,2} \zeta^{b_2} [1+\Oh(\zeta)], 
\quad C_{2,2} \in \bbC \backslash \{0\}, \; D \in \bbC.    
\end{split}
\end{align}
Once again, the leading order contribution as $\zeta \to 0$ stems from the $n=0$ term under the sum in \eqref{3.17}.  

For the second order poles at $\omega \in b_1+\bbN_0$, one now obtains
\begin{align}
& \f{1}{2 \pi i} \int_{\cC_{b_1,b_2}} d \omega \, \f{\Gamma(b_1-\omega) \Gamma(b_2-\omega)}{\prod_{j=2}^q 
\Gamma(1-b_j+\omega)} \zeta^{\omega}      \no \\ 
& \quad = - \sum_{n \in \bbN_0} {\rm Res\,}_{\omega= b_1+n} \bigg(\f{\Gamma(b_1-\omega) \Gamma(b_2-\omega)}{\prod_{j=3}^q \Gamma(1-b_j+\omega)} \zeta^{\omega}\bigg)     \no \\ 
& \quad = - \sum_{n \in \bbN_0} \lim_{\omega \to b_1+n} \f{d}{d\omega} \bigg[
\f{(\omega - b_1 - n)^2 \Gamma(b_1-\omega) \Gamma(b_2-\omega)}{\prod_{j=3}^q \Gamma(1-b_j+\omega)} \zeta^{\omega}\bigg]    \no \\
& \quad = - \lim_{\omega \to b_1} \f{d}{d\omega} \bigg[
\f{(\omega - b_1)^2 \Gamma(b_1-\omega) \Gamma(b_2-\omega)}{\prod_{j=3}^q \Gamma(1-b_j+\omega)} \zeta^{\omega}\bigg]     \\
& \hspace*{2.5mm} \underset{\zeta \to 0}{=}  \f{(-1)^{1+b_1-b_2}}{\Gamma(1+b_1-b_2) \prod_{j=3}^q \Gamma(1+b_1-b_j)} \zeta^{b_1} \ln(\zeta) \underbrace{[1+\Oh(\zeta)]}_{\text{analytic at $\zeta=0$}}      \no \\
& \hspace*{6mm} \quad + D_{2,1,b_1} \zeta^{b_1} \underbrace{[1+\Oh(\zeta)]}_{\text{analytic at $\zeta=0$}}       \no \\
& \hspace*{2.5mm} \underset{\zeta \to 0}{=} C_{2,1} \zeta^{b_1} \ln(\zeta) [1 + \Oh(\zeta)] 
+ D_{2,1,b_1} \zeta^{b_1} [1 + \Oh(\zeta)], \quad C_{2,1} \in \bbC \backslash \{0\}, \; D_{2,1,b_1} \in \bbC.    \no
\end{align}
(We recall that $(b_1 - b_2) \in \bbN_0$.) Since $\Re(b_1) \geq \Re(b_2)$, the term $D_{2,1,b_1} \zeta^{b_1} [1 + \Oh(\zeta)]$ is subordinate to $C_{2,2} \zeta^{b_2} [1+\Oh(\zeta)]$ and hence omitted in $y_2$ in \eqref{3.11}. 

One can now iterate this procedure for all $y_3, \dots, y_{n_1}$. In particular, if
\begin{equation}
\Re(b_{\ell-1}) > \Re(b_{\ell})= \dots = \Re(b_{n_1}), \quad 1 \leq \ell \leq n_1,
\end{equation}
then for $n \in \bbN_0$, one obtains in connection with the solution $y_{n_1}$ 
\begin{align}
\begin{split} 
& \lim_{\omega \to b_{\ell} + n} \f{d^{n_1-\ell}}{d\omega^{n_1-\ell}} \bigg[(\omega - b_{\ell}-n)^{n_1-\ell+1}     \\
& \qquad \;\, \times \f{\Gamma(b_1-\omega) \dots \Gamma(b_{\ell-1}-\omega)
[\Gamma(b_{\ell}-\omega) \dots \Gamma(b_{n_1}-\omega)]}{\prod_{j=n_1+1}^q \Gamma(1-b_j+\omega)} \zeta^{\omega}
\bigg]     \\
& \hspace*{2.5mm} \underset{\zeta \to 0}{=} C_{n_1,\ell} \zeta^{b_{\ell}} [\ln(\zeta)]^{n_1 - \ell} [1 + \Oh(\zeta)] 
+ \zeta^{b_{\ell}} \sum_{k=\ell+1}^{n_1} D_{n_1,k,b_{\ell}} [\ln(\zeta)]^{n_1-k}[1 + \Oh(\zeta)],      \\
& \hspace*{5.08cm} C_{n_1,\ell} \in \bbC \backslash \{0\}, \; D_{n_1,k,b_{\ell}} \in \bbC, \; \ell+1 \leq k \leq n_1.    
\end{split} 
\end{align}

At this point one repeats the analysis for each of the remaining congruence classes in \eqref{3.4}. In particular, repeating the analog of \eqref{3.11} for the remaining congruence classes in \eqref{3.4}, and noting that solutions belonging to different congruence classes are automatically linearly independent, then yields a fundamental system of solutions of \eqref{3.9}. 

Summarizing the discussion in this section one arrives at the following result:

\begin{theorem} \lb{t3.1}
Assume the notation established in \eqref{3.4}--\eqref{3.9}. Then a fundamental system of solutions of the $q$th-order differential equation \eqref{3.9} is given as follows: 
\begin{align}
y_1(\zeta) &= G_{0,q}^{1,0}\left(\!\begin{array}{c} {\displaystyle } \\
{\displaystyle b_1; b_2,\dots,b_q}\end{array} \bigg\vert - \zeta \right)   \no \\
&= \f{\zeta^{b_1}}{\prod_{j'=1}^q \Gamma(1+b_1-b_{j'})} \, {}_0 F_{q-1}\bigg(\!\begin{array}{c} {\displaystyle } \\
{\displaystyle 1+b_1-b_2,\dots,1+b_1-b_q}\end{array} \bigg\vert\, \zeta \bigg),  \no \\
y_2(\zeta) &= G_{0,q}^{2,0}\left(\!\begin{array}{c} {\displaystyle } \\
{\displaystyle b_1, b_2; b_3\dots,b_q}\end{array} \bigg\vert \, \zeta \right),   \no \\
& \hspace*{-6mm} \vdots    \no \\
y_{n_1}(\zeta) &= G_{0,q}^{n_1,0}\left(\!\begin{array}{c} {\displaystyle } \\
{\displaystyle b_1,\dots,b_{n_1}; b_{n_1+1},\dots,b_q}\end{array} \bigg\vert\, (-1)^{n_1} \zeta \right),     \no \\
y_{n_1+1}(\zeta) &= G_{0,q}^{1,0}\left(\!\begin{array}{c} {\displaystyle } \\
{\displaystyle b_{n_1+1}; b_1,\dots,b_{n_1},b_{n_1+2},\dots,b_q}\end{array} \bigg\vert - \zeta \right)     \no \\
&= \f{\zeta^{b_{n_1+1}}}{\prod_{j'=1}^q \Gamma(1+b_{n_1+1}-b_{j'})}     \no \\
& \quad \times {}_0 F_{q-1}\bigg(\!\begin{array}{c} {\displaystyle } \\
{\displaystyle 1+b_{n_1+1}-b_2,\dots,1+b_{n_1+1}-b_q}\end{array} \bigg\vert\, \zeta \bigg),  \no \\
y_{n_1+2}(\zeta) &= G_{0,q}^{2,0}\left(\!\begin{array}{c} {\displaystyle } \\
{\displaystyle b_{n_1+1}, b_{n_1+2}; b_1,\dots,b_{n_1},b_{n_1+3},\dots,b_q}\end{array} \bigg\vert\, \zeta \right),     \no \\
& \hspace*{-6mm} \vdots     \\
y_{n_1+n_2}(\zeta) &= G_{0,q}^{n_2,0}\left(\!\begin{array}{c} {\displaystyle } \\
{\displaystyle b_{n_1+1}, b_{n_1+2}, \dots, b_{n_1+n_2}; b_{n_1+n_2+1},\dots,b_q}\end{array} \bigg\vert\, (-1)^{n_2} \zeta \right),      \no \\
& \hspace*{-6mm} \vdots    \no \\
& \hspace*{-2cm} y_{n_1+n_2+ \cdots +n_{r-1}+1}(\zeta)       \no \\ 
& \hspace*{-1.5cm} = G_{0,q}^{1,0}\left(\!\begin{array}{c} {\displaystyle } \\
{\displaystyle b_{n_1+ \cdots + n_{r-1}+1}; b_1, \dots, b_{n_1+ \cdots + n_{r-1}}, b_{n_1+ \cdots + n_{r-1}+2},\dots,b_q}\end{array} \bigg\vert\, - \zeta \right)     \no \\
&  \hspace*{-1.5cm} = \f{\zeta^{b_{n_1+ \cdots + n_{r-1}+1}}}{\prod_{j'=1}^q \Gamma(1+b_{n_1+ \cdots + n_{r-1}+1} -b_{j'})}     \no \\
&  \hspace*{-1.05cm} \times {}_0 F_{q-1}\bigg(\!\begin{array}{c} {\displaystyle } \\
{\displaystyle 1+b_{n_1+ \cdots + n_{r-1}+1}-b_2,\dots,1+b_{n_1+ \cdots + n_{r-1}+1}-b_q}\end{array} \bigg\vert\, \zeta \bigg),  \no \\
& \hspace*{-2cm} y_{n_1+n_2+ \cdots +n_{r-1}+2}(\zeta)       \no \\ 
& \hspace*{-1.5cm} = G_{0,q}^{2,0}\bigg(\!\begin{array}{c} {\displaystyle } \\
{\displaystyle b_{n_1+ \cdots + n_{r-1}+1}, b_{n_1+ \cdots + n_{r-1}+2}; b_1, \dots, b_{n_1+ \cdots + n_{r-1}},}
\end{array}     \no \\ 
& \hspace*{3.1cm} \begin{array}{c} {\displaystyle } \\ 
{\displaystyle b_{n_1+ \cdots + n_{r-1}+3},\dots,b_q}\end{array} \bigg\vert\, \zeta \bigg),     \no \\
& \hspace*{-6mm} \vdots    \no \\ 
& \hspace*{-2cm} y_{n_1+n_2+ \cdots +n_{r-1}+n_r}(\zeta) \equiv y_q(\zeta)       \no \\ 
& \hspace*{-1.5cm} = G_{0,q}^{n_r,0}\left(\!\begin{array}{c} {\displaystyle } \\
{\displaystyle b_{n_1+ \cdots + n_{r-1}+1}, \dots, b_{n_1+ \cdots + n_{r-1}+n_r} \equiv b_q; b_1, \dots, b_{n_1+ \cdots + n_{r-1}}}\end{array} \bigg\vert\, (-1)^{n_r} \zeta \right);     \no \\
& \hspace*{10.35cm} \zeta \in \bbC \backslash \{0\}.     \no    
\end{align}
\end{theorem}

\section{The Generalized Eigenvalue Problem for Euler's Differential Equation} \lb{s4}

To fix our notation we denote by $\tau_N(\lambda_1,\dots,\lambda_N)$ the $N$th-order Euler differential expression in the form
\begin{equation}
\tau_N(\lambda_1,\dots,\lambda_N) = z^{-N} \prod_{j=1}^N \bigg[z\f{d}{dz} - \lambda_j\bigg], \quad 
\lambda_j \in \bbC, \, 1 \leq j \leq N, \; z \in \bbC\backslash\{0\}
\lb{4.1} 
\end{equation}
(noting the commutativity of the factors $[z(d/dz) - \lambda_j]$, $1\leq j \leq N$, on the right-hand side of \eqref{4.1}), in the complex domain, and  in the following focus on the generalized eigenvalue problem associated with \eqref{4.1}, that is, 
\begin{equation}
\tau_N(\lambda_1,\dots,\lambda_N) y(\mu,z)= \mu y(\mu,z), \quad \lambda_j \in \bbC, \, 1 \leq j \leq N, \; \mu \in \bbC, \; 
z \in \bbC\backslash\{0\},   \lb{4.2} 
\end{equation}
with $\mu \in \bbC$ the spectral parameter.

To make the connection of \eqref{4.2} with Meijer's $G$-functions we start by noting the following fact:

\begin{theorem} \lb{t4.1}
Consider the generalized eigenvalue problem for the $N$th-order Euler differential equation \eqref{4.2}. 
Upon introducing the change of variables
\begin{equation}
\zeta = \mu (z/N)^N, \quad \eta(\zeta) = y(\mu,z),     \lb{4.3}
\end{equation}
the generalized eigenvalue problem \eqref{4.2} turns into
\begin{equation}
\Bigg[\prod_{j=1}^N \bigg(\zeta \f{d}{d\zeta} - \f{\lambda_j}{N}\bigg) - \zeta\Bigg] \eta(\zeta) = 0, 
\quad \zeta \in \bbC\backslash\{0\},     \lb{4.4}
\end{equation}
the differential equation for 
$G_{0,N}^{m,0}\Big(\!\begin{array}{c} {\scriptstyle } \\
{\scriptstyle \lambda_1/N, \dots, \lambda_m/N; \lambda_{m+1}, \dots, \lambda_N/N}\end{array} 
\Big\vert\, (-1)^m \zeta\Big)$, $1 \leq m \leq N$, $\zeta \in \bbC\backslash\{0\}$. 

In particular, assuming the generic case where
\begin{equation}
[\lambda_j - \lambda_{j'}]/N \in \bbC \backslash \bbZ, \quad 1 \leq j, j' \leq N, \; j \neq j',    \lb{4.5}
\end{equation}
then 
\begin{align}
\eta_j(\zeta) &= G_{0,N}^{1,0}\left(\!\begin{array}{c} {\displaystyle } \\
{\displaystyle \lambda_j/N;\lambda_1/N,\dots,\lambda_{j-1}/N, \lambda_{j+1}/N,\dots,\lambda_N/N}\end{array} \bigg\vert -\zeta \right)    \no \\
& = \f{\zeta^{\lambda_j/N}}{\prod_{j'=1}^N \Gamma\Big(1+\f{\lambda_j-\lambda_{j'}}{N}\Big)}     \lb{4.7} \\
& \quad \times {}_0 F_{N-1}\bigg(\!\begin{array}{c} {\displaystyle } \\
{\displaystyle 1+[(\lambda_j-\lambda_1)/N],\dots,1+[(\lambda_{j}-\lambda_{j-1})/N],}\end{array} 
\no \\ 
& \hspace*{2.4cm} \begin{array}{c} {\displaystyle } \\
{\displaystyle 1+[(\lambda_j-\lambda_{j+1})/N],\dots,1+[(\lambda_j-\lambda_N)/N]}\end{array} \bigg\vert\, \zeta \bigg),  
\no \\
& \hspace*{6.95cm} 1 \leq j \leq N, \; \zeta \in \bbC\backslash\{0\},     \no
\end{align}
represents a fundamental system of solutions of \eqref{4.4}, subject to \eqref{4.5}. 

Finally,
\begin{align}
y_j(\mu,z) &= G_{0,N}^{1,0}\left(\!\begin{array}{c} {\displaystyle } \\
{\displaystyle \lambda_j/N;\lambda_1/N,\dots,\lambda_{j-1}/N, \lambda_{j+1}/N,\dots,\lambda_N/N}\end{array} \bigg\vert - \mu (z/N)^N\right)   \no \\
& = \f{\mu^{\lambda_j/N} (z/N)^{\lambda_j}}{\prod_{j'=1}^N \Gamma\Big(1+\f{\lambda_j-\lambda_{j'}}{N}\Big)}      \lb{4.8} \\
& \quad \times {}_0 F_{N-1}\bigg(\!\begin{array}{c} {\displaystyle } \\
{\displaystyle 1+[(\lambda_j-\lambda_1)/N],\dots,1+[(\lambda_{j}-\lambda_{j-1})/N],}\end{array} 
\no \\ 
& \hspace*{2.55cm} \begin{array}{c} {\displaystyle } \\
{\displaystyle 1+[(\lambda_j-\lambda_{j+1})/N],\dots,1+[(\lambda_j-\lambda_N)/N]}\end{array} \bigg\vert\, \mu (z/N)^N\bigg),  \no \\
& \hspace*{8.4cm} 1 \leq j \leq N, \; \zeta \in \bbC\backslash\{0\},    \no 
\end{align}
represents a fundamental system of solutions of the generalized Euler eigenvalue problem \eqref{4.2}, subject to \eqref{4.5}.
\end{theorem}
\begin{proof}
While \eqref{4.2}--\eqref{4.4} follow from elementary computations, \eqref{4.5}--\eqref{4.7} are a consequence of \eqref{2.16}, \eqref{2.17}. Finally, \eqref{4.8} is clear from \eqref{4.3} and \eqref{4.7}. 
\end{proof}

\begin{example} \lb{e4.2} Let $N=1$, $\lambda_1 \in \bbC$, then 
\begin{equation}
\tau_1(\lambda_1) = z^{-1} \bigg[z\f{d}{dz} - \lambda_1\bigg], \quad z \in \bbC\backslash\{0\},     \lb{4.9} 
\end{equation}
and the solution of 
\begin{equation}
\tau_1(\lambda_1) y(\mu,z)= \mu y(\mu,z), \quad \mu \in \bbC, \; z \in \bbC\backslash\{0\},   \lb{4.10} 
\end{equation}
is given by 
\begin{equation}
y_1(\mu,z) = G_{0,1}^{1,0}\left(\!\begin{array}{c} {\displaystyle } \\
{\displaystyle \lambda_1}\end{array} \bigg\vert -\mu z \right) 
= (- \mu z)^{\lambda_1} {}_0 F_0 \left(\!\begin{array}{c} {\displaystyle } \\
{\displaystyle  }\end{array} \bigg\vert \, \mu z \right)  
= (-\mu z)^{\lambda_1} e^{\mu z}, \quad z \in \bbC.      \lb{4.11}
\end{equation}
\end{example}

To treat the non-generic cases where condition \eqref{4.5} is violated, we now adapt the discussion in Section \ref{s3} to the present situation. We start by introducing some notation and decompose the set $\{\lambda_1/N,\dots,\lambda_N/N\}$ into congruence classes ${\rm mod}\,\bbZ$ as follows: 
\begin{align}
\begin{split} 
& \{\lambda_1/N,\dots,\lambda_N/N\} = \{\underbrace{\lambda_1/N,\dots,\lambda_{n_1}/N}_{\text{$n_1$ terms}}\} \cup 
\{\underbrace{\lambda_{n_1+1}/N,\dots,\lambda_{n_1+n_2}/N}_{\text{$n_2$ terms}}\} \cup \cdots      \lb{4.12} \\
& \quad \cdots \cup \{\underbrace{\lambda_{n_1+\cdots + n_{r-1}+1}/N,\dots,\lambda_{n_1+\cdots +n_{r-1} + n_r}/N \equiv \lambda_N/N}_{\text{$n_r$ terms}}\},  
\end{split} \\
& \sum_{s=1}^r n_s = N,    \lb{4.13} \\
\begin{split}
& \Re(\lambda_1) \geq \cdots \geq \Re(\lambda_{n_1}), \\
& \Re(\lambda_{n_1+1}) \geq \cdots \geq \Re(\lambda_{n_1+n_2}), \\
& \quad \vdots      \lb{4.14} \\
&  \Re(\lambda_{n_1+\cdots n_{r-1}+1}) \geq \cdots \geq \Re(\lambda_{n_1+\cdots n_{r-1}+n_r}), 
\end{split} \\
\begin{split} 
& [\lambda_j - \lambda_{j'}]/N \in \bbZ, \; 1 \leq j, j' \leq n_1, \\
& [\lambda_{n_1+j} - \lambda_{n_1+j'}]/N \in \bbZ, \; 1 \leq j, j' \leq n_2,  \\
& \quad \vdots    \lb{4.15} \\
& [\lambda_{n_1+\cdots+n_{r-1}+j} - \lambda_{n_1+\cdot+n_{r-1}+j'}]/N \in \bbZ, \; 1 \leq j, j' \leq n_r.
\end{split}
\end{align}

We also recall that 
\begin{equation}
G_{0,N}^{m,0}\left(\!\begin{array}{c} {\displaystyle } \\
{\displaystyle \lambda_1/N,\dots,\lambda_m/N; \lambda_{m+1}/N,\dots,\lambda_N/N}\end{array} \bigg\vert\, (-1)^m \zeta \right), \quad 1 \leq m \leq N, \lb{4.16}
\end{equation}
satisfies the ordinary differential equation of order $N$, 
\begin{equation}
\Bigg[\prod_{j=1}^N \bigg(\zeta \f{d}{d\zeta} - \f{\lambda_j}{N}\bigg) - \zeta \Bigg] \eta(\zeta) = 0, 
\quad \zeta \in \bbC\backslash\{0\}.       \lb{4.17} 
\end{equation}

Given these preparations, Theorem \ref{t3.1} yields the following result:

\begin{theorem} \lb{t4.3}
Assume the notation established in \eqref{4.12}--\eqref{4.17}. Then a fundamental system of solutions of the $N$th-order differential equation \eqref{4.17} is given as follows: 
\begin{align}
\eta_1(\zeta) &= G_{0,N}^{1,0}\left(\!\begin{array}{c} {\displaystyle } \\
{\displaystyle \f{\lambda_1}{N}; \f{\lambda_2}{N},\dots,\f{\lambda_N}{N}}\end{array} \Bigg\vert - \zeta \right)   \no \\
&= \f{\zeta^{\lambda_1/N}}{\prod_{j'=1}^N \Gamma\Big(1+\f{\lambda_1-\lambda_{j'}}{N}\Big)} \, {}_0 F_{N-1}\Bigg(\!\begin{array}{c} {\displaystyle } \\
{\displaystyle 1+\f{\lambda_1-\lambda_2}{N},\dots,1+\f{\lambda_1-\lambda_N}{N}}\end{array} \Bigg\vert\, \zeta \Bigg),  \no \\
\eta_2(\zeta) &= G_{0,N}^{2,0}\left(\!\begin{array}{c} {\displaystyle } \\
{\displaystyle \f{\lambda_1}{N}, \f{\lambda_2}{N}; \f{\lambda_3}{N}\dots,\f{\lambda_N}{N}}\end{array} \Bigg\vert \, \zeta \right),   \no \\
& \hspace*{-6mm} \vdots    \no \\
\eta_{n_1}(\zeta) &= G_{0,N}^{n_1,0}\left(\!\begin{array}{c} {\displaystyle } \\
{\displaystyle \f{\lambda_1}{N},\dots,\f{\lambda_{n_1}}{N}; \f{\lambda_{n_1+1}}{N},\dots,\f{\lambda_N}{N}}\end{array} \Bigg\vert\, (-1)^{n_1} \zeta \right),     \no \\
\eta_{n_1+1}(\zeta) &= G_{0,N}^{1,0}\left(\!\begin{array}{c} {\displaystyle } \\
{\displaystyle \f{\lambda_{n_1+1}}{N}; \f{\lambda_1}{N},\dots,\f{\lambda_{n_1}}{N},\f{\lambda_{n_1+2}}{N},\dots,\f{\lambda_N}{N}}\end{array} \Bigg\vert - \zeta \right)     \no \\
&= \f{\zeta^{\lambda_{n_1+1}/N}}{\prod_{j'=1}^N \Gamma\Big(1+\f{\lambda_{n_1+1}-\lambda_{j'}}{N}\Big)}     \no \\
& \quad \times {}_0 F_{N-1}\Bigg(\!\begin{array}{c} {\displaystyle } \\
{\displaystyle 1+\f{\lambda_{n_1+1}-\lambda_2}{N},\dots,1+\f{\lambda_{n_1+1}-\lambda_N}{N}}\end{array} \Bigg\vert\, \zeta \Bigg),  \no \\
\eta_{n_1+2}(\zeta) &= G_{0,N}^{2,0}\left(\!\begin{array}{c} {\displaystyle } \\
{\displaystyle \f{\lambda_{n_1+1}}{N}, \f{\lambda_{n_1+2}}{N}; \f{\lambda_1}{N},\dots,\f{\lambda_{n_1}}{N},\f{\lambda_{n_1+3}}{N},\dots,\f{\lambda_N}{N}}\end{array} \Bigg\vert\, \zeta \right),     \no \\
& \hspace*{-6mm} \vdots     \lb{4.17a} \\
\eta_{n_1+n_2}(\zeta) &= G_{0,N}^{n_2,0}\left(\!\begin{array}{c} {\displaystyle } \\
{\displaystyle \f{\lambda_{n_1+1}}{N}, \f{\lambda_{n_1+2}}{N}, \dots, \f{\lambda_{n_1+n_2}}{N}; \f{\lambda_{n_1+n_2+1}}{N},\dots,\f{\lambda_N}{N}}\end{array} \Bigg\vert\, (-1)^{n_2} \zeta \right),      \no \\
& \hspace*{-6mm} \vdots    \no \\
& \hspace*{-2cm} \eta_{n_1+n_2+ \cdots +n_{r-1}+1}(\zeta)       \no \\ 
& \hspace*{-1.5cm} = G_{0,N}^{1,0}\left(\!\begin{array}{c} {\displaystyle } \\
{\displaystyle \f{\lambda_{n_1+ \cdots + n_{r-1}+1}}{N}; \f{\lambda_1}{N}, \dots, \f{\lambda_{n_1+ \cdots + n_{r-1}}}{N}, \f{\lambda_{n_1+ \cdots + n_{r-1}+2}}{N},\dots,\f{\lambda_N}{N}}\end{array} \Bigg\vert\, - \zeta \right)     \no \\
&  \hspace*{-1.5cm} = \f{\zeta^{\lambda_{n_1+ \cdots + n_{r-1}+1}/N}}{\prod_{j'=1}^N \Gamma\Big(1+\f{\lambda_{n_1+ \cdots + n_{r-1}+1} -\lambda_{j'}}{N}\Big)}     \no \\
&  \hspace*{-1.05cm} \times {}_0 F_{N-1}\Bigg(\!\begin{array}{c} {\displaystyle } \\
{\displaystyle 1+\f{\lambda_{n_1+ \cdots + n_{r-1}+1}-\lambda_2}{N},\dots,1+\f{\lambda_{n_1+ \cdots + n_{r-1}+1}-\lambda_N}{N}}\end{array} \Bigg\vert\, \zeta \Bigg),  \no \\
& \hspace*{-2cm} \eta_{n_1+n_2+ \cdots +n_{r-1}+2}(\zeta)       \no \\ 
& \hspace*{-1.5cm} = G_{0,N}^{2,0}\Bigg(\!\begin{array}{c} {\displaystyle } \\
{\displaystyle \f{\lambda_{n_1+ \cdots + n_{r-1}+1}}{N}, \f{\lambda_{n_1+ \cdots + n_{r-1}+2}}{N}; 
\f{\lambda_1}{N}, \dots, \f{\lambda_{n_1+ \cdots + n_{r-1}}}{N},}
\end{array}     \no \\ 
& \hspace*{4.8cm} \begin{array}{c} {\displaystyle } \\ 
{\displaystyle \f{\lambda_{n_1+ \cdots + n_{r-1}+3}}{N},\dots,\f{b_N}{N}}\end{array} \Bigg\vert\, \zeta \Bigg),     \no \\
& \hspace*{-6mm} \vdots    \no \\ 
& \hspace*{-2cm} \eta_{n_1+n_2+ \cdots +n_{r-1}+n_r}(\zeta) \equiv y_N(\zeta)       \no \\ 
& \hspace*{-1.5cm} = G_{0,N}^{n_r,0}\Bigg(\!\begin{array}{c} {\displaystyle } \\
{\displaystyle \f{\lambda_{n_1+ \cdots + n_{r-1}+1}}{N}, \dots, \f{\lambda_{n_1+ \cdots + n_{r-1}+n_r}}{N} 
\equiv \f{\lambda_N}{N};} 
\end{array}     \no \\ 
& \hspace*{3.7cm} \begin{array}{c} {\displaystyle } \\ 
{\displaystyle \f{\lambda_1}{N}, \dots, \f{\lambda_{n_1+ \cdots + n_{r-1}}}{N}}\end{array} \Bigg\vert\, (-1)^{n_r} \zeta \Bigg);  \quad 
\zeta \in \bbC \backslash \{0\}.    \no 
\end{align}

In addition, the replacement $\zeta \longrightarrow \mu (z/N)^N$ in \eqref{4.17a} yields a fundamental system of solutions $y_j(\mu,\dott)$, $1 \leq j \leq N$, of the generalized Euler eigenvalue problem \eqref{4.2}. 
\end{theorem}

\begin{example} \lb{e4.4} Let $N=2$, $\lambda_j \in \bbC$, $j=1,2$, then 
\begin{equation}
\tau_2(\lambda_1,\lambda_2) = z^{-2} \bigg[z\f{d}{dz} - \lambda_1\bigg]\bigg[z\f{d}{dz} - \lambda_2\bigg], \quad z \in \bbC\backslash\{0\},     \lb{4.18} 
\end{equation}
and the solutions of 
\begin{equation}
\tau_2(\lambda_1,\lambda_2) y(\mu,z)= \mu y(\mu,z), \quad \mu \in \bbC, \; z \in \bbC\backslash\{0\},   \lb{4.19} 
\end{equation}
are given as follows: \\
Case $(i)$: If $[(\lambda_1-\lambda_2)/2] \in \bbC \backslash \bbZ$, then 
\begin{align}
\begin{split} 
y_1(\mu,z) &= G_{0,2}^{1,0}\left(\!\begin{array}{c} {\displaystyle } \\
{\displaystyle \f{\lambda_1}{2}; \f{\lambda_2}{2}}\end{array} \Bigg\vert -\mu z^2/4 \right)      \\
&= \f{\big(\mu^{1/2} z/2)^{\lambda_1}}{\Gamma\Big(1 + \f{\lambda_1-\lambda_2}{2}\Big)} 
{}_0 F_1 \left(\!\begin{array}{c} {\displaystyle } \\
{\displaystyle 1 + \f{\lambda_1 - \lambda_2}{2}}\end{array} \Bigg\vert \, \mu z^2/4 \right)        \\
&= \big(\mu^{1/2} z/2\big)^{(\lambda_1+\lambda_2)/2} I_{(\lambda_1-\lambda)/2}\big(\mu^{1/2}z\big),   \lb{4.20} \\
y_2(\mu,z) &= G_{0,2}^{1,0}\left(\!\begin{array}{c} {\displaystyle } \\
{\displaystyle \f{\lambda_2}{2}; \f{\lambda_1}{2}}\end{array} \Bigg\vert -\mu z^2/4 \right)      \\
&= \f{\big(\mu^{1/2} z/2)^{\lambda_1}}{\Gamma\Big(1 + \f{\lambda_2-\lambda_1}{2}\Big)} 
{}_0 F_1 \left(\!\begin{array}{c} {\displaystyle } \\
{\displaystyle 1 + \f{\lambda_2 - \lambda_1}{2}}\end{array} \Bigg\vert \, \mu z^2/4 \right)        \\
&= \big(\mu^{1/2} z/2\big)^{(\lambda_1+\lambda_2)/2} I_{(\lambda_2-\lambda_1)/2}\big(\mu^{1/2}z\big); \quad z \in \bbC\backslash\{0\}. 
\end{split}
\end{align}
Case $(ii)$: If $[(\lambda_1-\lambda_2)/2] \in \bbZ$, then 
\begin{align}
\begin{split} 
y_1(\mu,z) &= G_{0,2}^{1,0}\left(\!\begin{array}{c} {\displaystyle } \\
{\displaystyle \f{\lambda_1}{2}; \f{\lambda_2}{2}}\end{array} \Bigg\vert -\mu z^2/4 \right)      \\
&= \f{\big(\mu^{1/2} z/2)^{\lambda_1}}{\Gamma\Big(1 + \f{\lambda_1-\lambda_2}{2}\Big)} 
{}_0 F_1 \left(\!\begin{array}{c} {\displaystyle } \\
{\displaystyle 1 + \f{\lambda_1 - \lambda_2}{2}}\end{array} \Bigg\vert \, \mu z^2/4 \right)        \\
&= \big(\mu^{1/2} z/2\big)^{(\lambda_1+\lambda_2)/2} I_{(\lambda_1-\lambda)/2}\big(\mu^{1/2}z\big),   \lb{4.21} \\
y_2(\mu,z) &= G_{0,2}^{2,0}\left(\!\begin{array}{c} {\displaystyle } \\
{\displaystyle \f{\lambda_2}{2}, \f{\lambda_1}{2}}\end{array} \Bigg\vert -\mu z^2/4 \right)      \\
&= 2 \big(\mu^{1/2} z/2\big)^{(\lambda_1+\lambda_2)/2} K_{(\lambda_1-\lambda_2)/2}\big(\mu^{1/2}z\big); \quad z \in \bbC\backslash\{0\}. 
\end{split}
\end{align}
Here we used \cite[No.~9.2.2]{Lu75} in case $(i)$ and \cite[No.~8.4.8]{Lu75} in case $(ii)$.
\end{example}

\section{Some Remarks on the Nonhomogeneous Euler Differential Equation} \lb{s5}

In our final section we add some facts in connection with the inhomogeneous $N$th-order Euler differential equations, that is, we consider the general solution of the nonhomogeneous $N$th-order Euler differential equation.
 
\begin{proposition} \lb{p5.1} 
Let $N \in \bbN$, $\lambda_j \in \bbC$, $1 \leq j \leq N$, and $f \in C(\bbC)$. Then the general solution of the nonhomogeneous $N$th-order Euler differential equation
\begin{equation}
\tau_N(\lambda_1,\dots,\lambda_N) y(z) = z^{-N} f(z), \quad \lambda_j \in \bbC, \; 
z \in \bbC\backslash\{0\},     \lb{5.1}
\end{equation}
is of the following form\footnote{For simplicity of notation in \eqref{5.3} we incorporated the overall factor $z^{-N}$ on the left-hand side of \eqref{5.1} also into its nonhomogeneous term on the right-hand side.}: Grouping  $\{\lambda_1,\dots,\lambda_N\}$ into pairwise distinct elements counting multiplicity, 
\begin{align}
\begin{split} 
&\{\lambda_1,\dots,\lambda_N\} = \bigcup_{k = 1}^{K} 
\{\underbrace{\lambda_{j_k},\dots,\lambda_{j_k}}_{m_{j_k} \, times}\}, \quad 
\lambda_{j_k} \neq \lambda_{j_{k'}}, \; 1 \leq k, k' \leq K, \; k \neq k', \; 1 \leq K \leq N,     \lb{5.2} \\ 
& \sum_{k=1}^K m_{j_k} = N,
\end{split}
\end{align}
the general solution of \eqref{5.1} is given by
\begin{align}
y(z) &= \sum_{k=1}^K \sum_{m=0}^{m_{j_k}-1} \big\{C_{k,m} z^{\lambda_{j_k}} [\ln(z)]^m\big\}  
+ \int_0^1 ds_1 \cdots \int_0^1 ds_N \, f\Bigg(z \prod_{j=1}^N s_j\Bigg) 
\prod_{\ell=1}^N s_{\ell}^{-\lambda_{\ell} - 1},     \no \\
& \hspace*{3cm} C_{k,m} \in \bbC, \; 0 \leq m \leq m_{j_k} - 1, \, 1 \leq k \leq K, \; 
z \in \bbC\backslash\{0\}.  \lb{5.3} 
\end{align}
\end{proposition}
\begin{proof}
That 
\begin{align}
\begin{split} 
& \sum_{k=1}^K \sum_{m=0}^{m_{j_k}-1} \big\{C_{k,m} z^{\lambda_{j_k}} [\ln(z)]^m\big\},     \\
& \, C_{k,m} \in \bbC, \; 0 \leq m \leq m_{j_k} - 1, \, 1 \leq k \leq K, \; z \in \bbC\backslash\{0\},     \lb{5.4} 
\end{split}
\end{align} 
is the general solution of the homogeneous $N$th-order Euler differential equation is well-known; it easily follows from taking repeated $\lambda$-derivatives of the underlying indicial (or characteristic) polynomial
\begin{equation}
P_N(\lambda) = \prod_{j=1}^N (\lambda - \lambda_j),    \lb{5.5} 
\end{equation}
noting that ($z \in\bbC\backslash\{0\}$)
\begin{align}
& \tau_N(\lambda_1,\dots,\lambda_N) \big(z^{\lambda}\big) = P_N(\lambda) \big(z^{\lambda + N}\big),    \no \\
& \tau_N(\lambda_1,\dots,\lambda_N) \big((\partial/\partial \lambda) z^{\lambda}\big) 
= \tau_N(\lambda_1,\dots,\lambda_N) \big(z^{\lambda} \ln(z)\big)    \no \\
& \quad = (\partial/\partial \lambda) \big[\tau_N(\lambda_1,\dots,\lambda_N) \big(z^{\lambda}\big)\big]
= [P_N'(\lambda) + P_N(\lambda) \ln(z)] z^{\lambda + N},     \lb{5.6} \\
& \quad \;\; \vdots   \no \\
& \quad \text{etc.}       \no 
\end{align}
Hence, it remains to show that 
\begin{equation}
\int_0^1 ds_1 \cdots \int_0^1 ds_N \, f\Bigg(z \prod_{j=1}^N s_j\Bigg) 
\prod_{\ell=1}^N s_{\ell}^{-\lambda_{\ell} - 1}, \quad z \in \bbC\backslash\{0\},      \lb{5.7} 
\end{equation}
is a particular solution of \eqref{5.1}. This in turn follows from iterating the elementary computation,
\begin{align}
& \prod_{j=1}^N \bigg[z\f{d}{dz} - \lambda_j\bigg] \Bigg[\int_0^1 ds_1 \cdots \int_0^1 ds_N \, f\Bigg(z \prod_{j=1}^N s_j\Bigg) \prod_{\ell=1}^N s_{\ell}^{-\lambda_{\ell} - 1}\Bigg]    \no \\
& \quad = \prod_{j=1}^N \bigg[z\f{d}{dz} - \lambda_j\bigg] \Bigg[z^{\lambda_1} \int_0^z d\zeta_1 \int_0^1 ds_2 
\cdots \int_0^1 ds_N \, f\Bigg(\zeta_1 \prod_{j=2}^N s_j\Bigg)    \no \\
& \hspace*{8.5cm}  \times \zeta_1^{- \lambda_1 -1} \prod_{\ell=2}^N s_{\ell}^{-\lambda_{\ell} - 1}\Bigg]    \no \\
& \quad = \prod_{j=2}^N \bigg[z\f{d}{dz} - \lambda_j\bigg] \Bigg[(\lambda_1 - \lambda_1) z^{\lambda_1} \int_0^z d\zeta_1 \int_0^1 ds_2 \cdots \int_0^1 ds_N \, f\Bigg(\zeta_1 \prod_{j=2}^N s_j\Bigg)      \no \\
& \hspace*{10cm}  \times \zeta_1^{- \lambda_1 -1} \prod_{\ell=2}^N s_{\ell}^{-\lambda_{\ell} - 1}    \no \\
& \hspace*{3.1cm} + z^{1+\lambda_1} \int_0^1 ds_2 \cdots \int_0^1 ds_N \, 
f\Bigg(z \prod_{j=2}^N s_j\Bigg) z^{-\lambda_1-1} \prod_{\ell=2}^N s_{\ell}^{-\lambda_{\ell} - 1}\Bigg]  \no \\
& \quad = \prod_{j=2}^N \bigg[z\f{d}{dz} - \lambda_j\bigg] \Bigg[\int_0^1 ds_2 \cdots \int_0^1 ds_N \, 
f\Bigg(z \prod_{j=2}^N s_j\Bigg) \prod_{\ell=2}^N s_{\ell}^{-\lambda_{\ell} - 1}\Bigg]  \no \\
& \quad = \prod_{j=2}^N \bigg[z\f{d}{dz} - \lambda_j\bigg] \Bigg[z^{\lambda_2} \int_0^z d\zeta_2 \int_0^1 ds_3 \cdots 
\int_0^1 ds_N \, f\Bigg(\zeta_2 \prod_{j=3}^N s_j\Bigg)     \no \\ 
& \hspace*{7.6cm} \times \zeta_2^{- \lambda_2 -1} \prod_{\ell=3}^N s_{\ell}^{-\lambda_{\ell} - 1}\Bigg]     \no \\
& \quad \;\; \vdots    \no \\
& \quad = \bigg[z \f{d}{dz} - \lambda_N\bigg] z^{\lambda_N} \int_0^z d\zeta_N \, f(\zeta_N) \zeta_N^{-\lambda_N - 1}  \no \\
& \quad = (\lambda_N - \lambda_N) z^{\lambda_N} \int_0^z d\zeta_N \, f(\zeta_N) \zeta_N^{-\lambda_N -1}
+ z^{1+\lambda_N} f(z) z^{-\lambda_N - 1}    \no \\ 
& \quad = f(z), \quad z \in \bbC\backslash\{0\}. 
\end{align}
Here the path from $0$ to $z$ in the $\zeta_j$-plane, $1\leq j \leq N$, can be chosen to be a straight line segment. 
\end{proof}

We note that the special case $N=2$ in Proposition \ref{p5.1} is mentioned in \cite[p.~202]{In56}. 
	
\medskip

\noindent {\bf Acknowledgments.} We are indebted to Mark Ashbaugh and Andrei Martinez-Finkelshtein for very helpful hints to the literature in connection with Meijer's $G$-function.


\end{document}